\documentclass[11pt]{amsart} 

\usepackage[english]{babel}
\usepackage[utf8]{inputenc}
\usepackage[margin=1in]{geometry}
\usepackage{amsmath}
\usepackage{amsthm}
\usepackage{amsfonts}
\usepackage{amssymb}
\usepackage[usenames,dvipsnames]{xcolor}
\usepackage{graphicx}
\usepackage[siunitx]{circuitikz}
\usepackage{tikz}
\usepackage[colorinlistoftodos, color=orange!50]{todonotes}
\usepackage{hyperref}
\usepackage[numbers, square]{natbib}
\usepackage{fancybox}
\usepackage{epsfig}
\usepackage{soul}
\usepackage[framemethod=tikz]{mdframed}
\usepackage[shortlabels]{enumitem}
\usepackage[version=4]{mhchem}
\usepackage{multicol}

\usepackage{amsthm}
\usepackage{mathtools}
\usepackage{graphicx}
\usepackage{tikz-cd}
\graphicspath{ {./images/} }

\newtheorem{theorem}{Theorem}
\newtheorem*{theorem*}{Theorem}
\newtheorem*{theoremA}{Theorem A}
\newtheorem{lemma}{Lemma}[section]
\newtheorem*{lemma*}{Lemma}

\newtheorem{corollary}{Corollary}[section]
\newtheorem*{corollary*}{Corollary}
\newtheorem{proposition}{Proposition}[section]
\newtheorem*{proposition*}{Proposition}

\theoremstyle{definition}
\newtheorem{definition}{Definition}

\theoremstyle{remark}
\newtheorem{remark}{Remark}[section]

\DeclareMathOperator{\Ad}{Ad}
\DeclareMathOperator{\Sl2}{SL_{2}}
\DeclareMathOperator{\sgn}{sgn}

\DeclareMathOperator{\PSl2}{PSL_{2}}

\begin{document}

\author{Dana Abou Ali}
\address{Mathematisches Institut, Endenicher Allee 60, 53115 Bonn, Germany}
\email{danaaabouali@gmail.com}

\thanks{I would like to express my deepest gratitude to Professor Valentin Blomer for introducing me to this topic and for his invaluable guidance. This paper was supported in part by Germany's Excellence Strategy grant EXC-2047/1 - 390685813}


\title{On the $L^{2}$-restriction norm problem for closed geodesics on the modular surface}

\begin{abstract}   Let $f$ be a Petersson normalized Hecke-Maass cusp form with spectral parameter $t\geq 2$ and let $\mathcal{C}_{D}$ be the union of closed geodesics in $\Sl2(\mathbb{Z})\setminus \mathbb{H}$ associated to a fundamental discriminant $D>0$. Following a suggestion by Sarnak in his letter to Reznikov, we express the restriction norm $||f|_{\mathcal{C}_{D}}||_{2}^{2}$ as a weighted sum of central values of L-functions using Waldspurger's formula. This allows us to get an unconditional improvement over the current bounds. 
 \end{abstract}


\setcounter{tocdepth}{2}  \maketitle

\maketitle

\section{Introduction}
\noindent This paper deals with a geodesic restriction problem: given a Riemann surface $(M,g)$, a normalized function $f \colon M\to \mathbb{C}$ and a closed geodesic $\ell \subseteq M$, we want to estimate $||f|_{\ell}||_{2}$ with respect to some parameter. The goal of this problem is, informally, to understand the extent to which $f$ can concentrate on small subsets of $M$ ( \cite{Marshall_2016}, \cite{Restriction-Laplace-Beltrami-Burq}). The following standard result was proved by Reznikov \cite{Reznikov-geodesic-restriction} and generalized by Burq, G\'erard and Tzvektov \cite{Restriction-Laplace-Beltrami-Burq}
\begin{theorem*}
If $M$ is a compact smooth Riemannian surface (without boundary), $\Delta$ is the Laplace operator associated to $g$ and $\ell \subseteq M$ a smooth curve then
\[
||f|_{\ell} ||_{2} \ll (1+r)^{\frac{1}{4}}
\]
for any eigenvector $f$ of $\Delta$ with eigenvalue $-r^{2}$, $r\geq 0$. 
\end{theorem*}
\noindent
It is possible to get better bounds by adding some restrictions on $M$, $\ell$ and $f$. For example, whenever we consider arithmetic objects, one can use techniques from number theory to get a smaller exponent (see \cite{Blomer_2013} or \cite{Ghosh-Reznikov-Sarnak} for example). In this paper, we take $M$ to be the modular surface $\Sl2 (\mathbb{Z})\setminus \mathbb{H}$, $f$ a Hecke-Maass cusp form and $\ell \coloneqq \mathcal{C}_{D}$ the union of all closed geodesics in $\Sl2(\mathbb{Z})\setminus \mathbb{H}$ associated to a fundamental discriminant $D>0$. 
\newline 
\newline
Marshall \cite{Marshall_2016} proved the following
\begin{theorem*}
Let $M$ be a compact congruence arithmetic hyperbolic surface arising from a quaternion division algebra over $\mathbb{Q}$, $\ell \subseteq M$ a geodesic segment of unit length and $f$ an $L^{2}$ normalized Hecke-Maass cusp form with spectral parameter $t \in \mathbb{R}$.
Then 
\begin{equation}
\label{eq: marshall bound}
||f|_{\ell}||_{2} \ll_{\varepsilon} t^{\frac{3}{14}+\varepsilon}.
\end{equation}
\end{theorem*}
\noindent Using \cite{bourgain2009geodesic} and the bound $||f||_{4} \ll_{\varepsilon} t^{\varepsilon}$ announced by Sarnak and Watson in \cite{sarnak-watson}, the exponent in (\ref{eq: marshall bound}) can be improved so that
\[
||f|_{\ell}||_{2} \ll_{\varepsilon} t^{\frac{1}{8}+\varepsilon}.
\]

\noindent Following a suggestion by Sarnak in a letter to Reznikov (see \cite{sarnak-reznikov}, p.6 example (c)), we prove  
\begin{theoremA}
\label{theorem A}
Let $f$ be a Hecke-Maass cusp form with $||f||_{2} = 1$ and Laplace eigenvalue $\lambda = 1/4+ t^{2}$ for some $t\geq 2$. Let $\mathcal{C}_{D}$ be the union of geodesics defined above. Then, 
\begin{equation}
    ||f|_{\mathcal{C}_{D}} ||_{2} \ll_{D,\varepsilon} t^{\theta+\varepsilon}
\end{equation}
where $\theta \geq 0$ is any bound towards the Ramanujan conjecture for Maass forms. 
\end{theoremA}
\noindent In particular, assuming the Ramanujan conjecture we would get $||f|_{\mathcal{C}_{D}}||_{2} \ll_{\varepsilon} t^{\varepsilon}$. Using the Kim-Sarnak bound \cite{kim-sarnak-bound}, we have the unconditional bound
\[
||f|_{\mathcal{C}_{D}} ||_{2} \ll_{D,\varepsilon} t^{\frac{7}{64}+\varepsilon}
\]
which is an improvement over the previously mentioned results.

\subsection*{Sketch of proof}
We start the proof by finding an appropriate expression for $||f|_{\mathcal{C}_{D}}||_{2}^{2}$ involving completed L-functions. To do so, we endow $\mathcal{C}_{D}$ with a compact abelian group structure by constructing a homeomorphism to a disjoint union of circles. Then, using representation theory of compact abelian groups we find an orthonormal basis of $L^{2}(\mathcal{C}_{D})$ which allows us to  write $||f|_{\mathcal{C}_{D}}||_{2}^{2}$ as a sum of squares of Fourier coefficients. By Waldspurger's formula, these Fourier coefficients are related to the values $\Lambda (1/2,f\times \theta_{\chi})$ for some Hecke characters $\chi$. In particular, to prove Theorem A it will be enough to show that
\[
\sum_{\psi\in \widehat{H^{+}(K)}}\sum_{n \in \mathbb{Z}}\frac{\gamma(1/2,f\times \theta_{\chi_{\psi,n}})}{\gamma(1,\Ad f)} L(1/2, f\times \theta_{\chi_{\psi,n}}) \ll_{\varepsilon,D} t^{2\theta + \varepsilon}
\]
where $K = \mathbb{Q}(\sqrt{D})$ and $\widehat{H^{+}(K)}$ is the dual of the ideal class group.
\newline
\newline
The quotient of gamma factors decays exponentially for $|n| \geq c_{D}t$ where $c_{D}$ is a constant to be determined later. In particular, we are essentially dealing with a finite sum. Using the approximate functional equation, we can express $L(1/2,f\times \theta_{\chi_{\psi,n}})$ as a sum in terms of the Fourier coefficients $\lambda_{f}(m)$ of $f$ and $\lambda_{\theta_{\chi_{\psi,n}}}(m)$ of $\theta_{\chi_{\psi,n}}$. By definition, $\lambda_{\theta_{\chi_{\psi,n}}}(m)$ consists of a sum over ideals in $K=\mathbb{Q}(\sqrt{D})$. However, after applying characters orthogonality to the $\widehat{H^{+}(K)}$ sum, we only need to consider sums over principal ideals with totally positive generators. In other words, the sum over ideals can be reduced to a sum over elements $\alpha$ in an appropriately chosen fundamental domain $\mathcal{F}_{D}$ for the action of the totally positive units on the totally positive integers of $K$.
\newline 
\newline 
Finally, we apply the Poisson summation formula to the $n$ sum to find some cancellations. As a result, we can further restrict our $\alpha$ sum to a small enough subset of generators whenever $|c_{D}t - n|$ is large. Applying a dyadic partition of unity allows us to minimize the effect of the terms in the $n$ sum for which $|c_{D}t - n|$ is small.

\section{Background material and setup}
\label{sec: background material}

\subsection{Real quadratic number fields} 
\label{real quadratic fields}

\noindent Throughout this paper, we will let $D>0$ be a fundamental discriminant and $K\coloneqq \mathbb{Q}(\sqrt{D})$ the corresponding real quadratic number field. We denote the ring of integers, group of units and narrow class group of $K$ by $\mathcal{O}_{K}$, $U(K)$ and $H^{+}(K)$ respectively.
\newline
\newline
As a consequence of Dirichlet's unit theorem, the group of totally positive units $U^{+}(K)$ is infinite cyclic. We denote by $\varepsilon_{D}$ the unique generator of $U^{+}(K)$ such that $\varepsilon_{D}>1$ as an element of $\mathbb{R}$. 
\newline
\newline
Finally, we consider the action of the group $U^{+}(K)$ on $\mathcal{O}_{K}^{+}$, the set of totally positive integers, via multiplication. The embedding $ \iota_{K} \colon \alpha \mapsto (\alpha, \alpha^{*})
$ of $K$ into $\mathbb{R}^{2}$ 
allows us to obtain a fundamental domain $\mathcal{F}_{D} \subseteq \mathbb{R}^{2}$ for this action. More precisely,

\begin{proposition}
\label{fundamental domain}
Let $\beta_{D} \coloneqq \frac{1+\sqrt{D}}{2}$ for $D\equiv 1 \mod 4$ and $\beta_{D} \coloneqq \frac{\sqrt{D}}{2}$ otherwise. If we see $K$ as a subset $\mathbb{R}^{2}$ using $\iota_{K}$, a fundamental domain $\mathcal{F}_{D}$ for the action of $U^{+}(K)$ on $\mathcal{O}_{K}^{+}$ is equal to
\[
\mathcal{F}_{D} \coloneqq C_{D} \cap L_{D}
\]
where $C_{D}$ is the cone
$
C_{D} \coloneqq \{ (x,y) \in (\mathbb{R}_{>0})^{2} \mid y \leq x < \varepsilon_{D}^{2} y \} 
$
and $L_{D}$ is the lattice
$
L_{D} \coloneqq \{ (a+b\beta_{D}, a+b\beta_{D}^{*}) \mid a,b \in \mathbb{Z} \}.
$
\end{proposition}

\subsection{Closed geodesics of the modular surface}

\noindent  As mentioned previously, we will be working with $\Sl2(\mathbb{Z})\setminus \mathbb{H}$, the orbit space of the action of subgroup $\Sl2(\mathbb{Z})$ on $\mathbb{H}$. In this subsection, we give a summary of the description of closed geodesics in $\Sl2(\mathbb{Z})\setminus \mathbb{H}$ developed in Section 1 of \cite{Sarnak-geodesics}. 
 
\begin{definition}
\begin{enumerate}[(i)]
\item An element $A \in \Sl2(\mathbb{Z})$ is said to be hyperbolic  if the corresponding M\"obius transformation has precisely two fixed points $x,x' \in \mathbb{R}$.
\item A hyperbolic element $A \in \Sl2(\mathbb{Z})$ is called primitive if we cannot find any $B \in \Sl2(\mathbb{Z})$ and $n \in \mathbb{Z}\setminus \{0,1,-1 \}$ with
$
A = B^{n}.
$
\end{enumerate}
\end{definition}

\noindent We have a right action of $\Sl2(\mathbb{Z})$ on hyperbolic elements via
$
A.B \coloneqq B^{-1}A B
$. This leads to an induced action of $\PSl2(\mathbb{Z})$ on primitive hyperbolic elements. 
\begin{definition}
Let $A \in \Sl2(\mathbb{Z})$ be hyperbolic and let $\begin{pmatrix} a & 0 \\ 0 & a^{-1} \end{pmatrix}$, $|a|>1$ be the unique matrix of this form which is conjugate to $A$. We define the norm of $A$ to be
$
N(A) = a^{2}.
$
\end{definition}

\noindent Let $A \in \Sl2(\mathbb{Z})$ be a primitive hyperbolic element and let $x,x' \in \mathbb{R}$ be the fixed points of the corresponding M\"obius transformation. Let $\tilde{\gamma_{A}} \subseteq \mathbb{H}$ be the geodesic joining $x$ and $x'$ and let $\gamma_{A}$ be its image in $\Sl2(\mathbb{Z})\setminus \mathbb{H}$.
Then, $\gamma_{A}$ is a closed geodesic in $\Sl2(\mathbb{Z})\setminus \mathbb{H}$ of length $\log(N(A))$. Moreover, this construction induces a bijection between closed geodesics in $\Sl2(\mathbb{Z})\setminus \mathbb{H}$ and conjugacy classes of primitive hyperbolic elements (up to identifying $A$ with $-A$ ).

\begin{definition}
A primitive indefinite quadratic form of discriminant $D>0$ is a quadratic form $Q(x,y)$ of the form
$
Q(x,y) = ax^{2}+bxy + cy^{2}
$
with $a,b,c \in \mathbb{Z}$, $\gcd(a,b,c) =1$ and $b^{2}-4ac = D$. 
\newline We use the shorthand notation $[a,b,c]$ to denote the quadratic form $Q(x,y) = ax^{2}+bxy + cy^{2}$.
\end{definition}

\noindent From now on, unless specified otherwise, all quadratic forms we consider will be primitive indefinite.
\newline 
\newline 
We have an $\Sl2(\mathbb{Z})$ action on the space of quadratic forms via change of variables. More precisely, given a quadratic form $Q(x,y)$ and $A \in \Sl2(\mathbb{Z})$, we define
$
(Q.A)(x,y) \coloneqq Q(A(x,y))
$
where $A(x,y)$ is the image of $(x,y)$ by the linear transformation induced by $A$. 
For any fundamental discriminant $D>0$, this action induces an action of $\PSl2(\mathbb{Z})$ on the space of quadratic forms of discriminant $D$.

\begin{proposition}
\label{quadratic forms- hyperbolic elemets}
Let $[a,b,c]$ be a quadratic form. Then, 
\begin{enumerate}[(i)]
    \item the stabilizer of $[a,b,c]$ by the action of $\PSl2(\mathbb{Z})$ is the image of the infinite cyclic subgroup with generator
$
M_{[a,b,c]} = \begin{pmatrix} \frac{x_{D}-by_{D}}{2} & -cy_{D} \\ ay_{D} & \frac{x_{D}+by_{D}}{2} \end{pmatrix}
$
where $D>0$ is the discriminant of $[a,b,c]$ and $(x_{D},y_{D}) \in (\mathbb{Z}_{\geq 0})^{2}$ is the fundamental solution of the Pell equation
$
x^{2} - Dy^{2} = 4
$,
\item the map $[a,b,c] \mapsto M_{[a,b,c]}$ gives a bijection of $\PSl2(\mathbb{Z})$-sets between quadratic forms and primitive hyperbolic elements in $\PSl2(\mathbb{Z})$. In particular, this bijection descends to a bijection of the corresponding equivalence classes and
\item for any fundamental discriminant $D>0$, the bijection $[a,b,c] \mapsto M_{[a,b,c]}$ restricts to a bijection between quadratic forms of discriminant $D>0$ and primitive hyperbolic elements of norm $\varepsilon_{D}^{2}$.
\end{enumerate}
\end{proposition}

\noindent As a consequence, we have an explicit description for the closed geodesics of $\Sl2(\mathbb{Z})\setminus \mathbb{H}$.
\begin{theorem}
Let $\gamma \subseteq \Sl2(\mathbb{Z}) \setminus \mathbb{H}$ be a closed geodesic and let $\tilde{\gamma}$ be a lift of $\gamma$ to $\mathbb{H}$. Let $\theta_{1},\theta_{2} \in \mathbb{R}$ be the endpoints of $\tilde{\gamma}$. Then,
\begin{enumerate}[(i)]
    \item there exists a fundamental discriminant $D>0$ such that $\gamma$ has length $\log(\varepsilon_{D}^{2})$,
    \item there exists a quadratic form $[a,b,c]$ of discriminant $D$ such that $M_{[a,b,c]}$ fixes the endpoints of $\tilde{\gamma}$. In this case, $\theta_{1}$ and $\theta_{2}$ are roots of the polynomial $ax^{2}+bx +c$ and in particular, $\theta_{1}, \theta_{2} \in \mathbb{Q}(\sqrt{D})$ are Galois conjugates.
\end{enumerate}
\end{theorem}

\noindent Finally, let $D>0$ be a fundamental discriminant and $K= \mathbb{Q}(\sqrt{D})$. There is a natural bijection between the narrow class group $H^{+}(K)$ of $K$ and the equivalence classes of quadratic forms of discriminant $D$ modulo $\Sl2(\mathbb{Z})$  (see the appendix in \cite{quadratic-forms-ideal}). As a consequence, we have the following
\begin{theorem}
\label{closed geodesics- class group}
There exists a bijection between the narrow class group $H^{+}(K)$ and the set of closed geodesics of $\Sl2(\mathbb{Z})\setminus \mathbb{H}$ of length $\log(\varepsilon_{D}^{2})$.
\end{theorem}

\begin{remark}
Alternatively, we could fix a set of primitive closed geodesics. In this case, whenever the quadratic form $[a,b,c]$ is in the same $\Sl2(\mathbb{Z})$-orbit as $[-a,-b,-c]$, the length of the geodesic would be $\log(\varepsilon_{D})$ instead of $2\log(\varepsilon_{D})$. However, thanks to the normalization in (\ref{eq: L2 norm explicit}), the restriction norm formula will not be affected. 
\end{remark}

\noindent We can now define the set
\[
\mathcal{C}_{D} \coloneqq \bigcup_{\gamma} \gamma
 \] 
where the union is taken over all closed geodesics $\gamma$ in $\Sl2(\mathbb{Z})\setminus \mathbb{H}$ of length $\log(\varepsilon_{D}^{2})$.

\subsection{Hecke characters}

In this paper, we will only be working with real quadratic number fields, and therefore, we define Hecke characters only in this particular context. For more details, we refer the reader to \cite{narkiewicz}, Chapter 7.
\begin{definition}[Hecke characters]
Let $K= \mathbb{Q}(\sqrt{D})$ be a real quadratic field with discriminant $D>0$ and let $\varepsilon_{D} >1 $ be a generator of $U^{+}(K)$. Consider a decomposition $H^{+}(K) \cong \mathbb{Z}/h_{1}\mathbb{Z} \oplus \hdots \oplus \mathbb{Z}/h_{s}\mathbb{Z}$ of $H^{+}(K)$ into cyclic groups and fix $J_{1},\hdots,J_{h}$ ideals in $K$ such that $[J_{i}]$ generates the cyclic factor $\mathbb{Z}/h_{i}\mathbb{Z}$ in $H^{+}(K)$.
Then, for any class group character $\psi$ and any $n\in \mathbb{Z}$ we define the Hecke character $\chi_{\psi,n}$ as
\[
\chi_{\psi,n}([\mathfrak{a}]) = \psi([\mathfrak{a}]) \Big| \frac{x}{x^{*}} \Big|^{\frac{\pi i n}{\log(\varepsilon_{D})}}
\]
for $\mathfrak{a} = xJ_{1}^{t_{1}}\hdots J_{s}^{t_{s}}$, $0\leq t_{i} < h_{i}$.
\end{definition}

\noindent Given a Hecke character $\chi_{\psi,m}$, we let $\theta_{\chi_{\psi,m}}$ be the corresponding theta function. This function is then a Maass form of level $D$, spectral parameter $\frac{ m \pi}{\log(\varepsilon_{D})}$ and nebentypus $\chi_{D}$, the Dirichlet character obtained using the Legendre symbol (for more details see \cite{bump}, Chapter 1.9). The $n$-th Fourier coefficient of this Maass form are given by
\[
\lambda_{\theta_{\chi_{\psi,m}}}(n) = \begin{cases} \frac{1}{2} \sum_{\mathbb{N}(\mathfrak{a})=|n|} \chi_{\psi,m}([\mathfrak{a}]) \\
\frac{1}{2i}  \sgn(n) \sum_{\mathbb{N}(\mathfrak{a})=|n|} \chi_{\psi,m}([\mathfrak{a}]).
\end{cases}
\]
Using the fact that Hecke characters are unitary, we can have the following estimate
\begin{lemma}
\label{fourier coef theta}
For a Hecke character $\chi$, we have
$
\lambda_{\theta_{\chi}}(n) \ll_{\varepsilon} |n|^{\varepsilon}$
for any $\varepsilon>0$.
\end{lemma}

\subsection{L-functions}
\label{rankin-selberg section}

 From now on, all Maass forms are assumed to be Hecke-Maass cusp forms. Given a Maass form $g$ and $n\in \mathbb{Z}$ we denote the $n$-th Fourier coefficient of $g$ by $\lambda_{g}(n)$.
\newline 
\newline 
In this paper we will be working with two types of L-functions: the Rankin-Selberg convolution $L(s,f\times \theta_{\chi})$ where $f$ is a Maass form of level 1 and $\chi$ is a Hecke character, and the adjoint square L-function $L(1,\Ad f)$ where $f$ is once again a Maass form. For the definitions and properties of these L-functions, we refer the reader to \cite{cornut-vatsal} (Section 1.1) and \cite{IwaniecKowalski} (p. 136--138). We briefly state a few results we will need in the proof of the main theorem.

\begin{definition}
\label{fancy gamma function}
For any $s\in \mathbb{C}$ and $x \in \mathbb{R}$ we define the function
\[
\gamma(s,f,x) \coloneqq  \pi^{-2s} \prod_{\pm} \Gamma\Big(\frac{s +i(t\pm\frac{\pi x}{\log(\varepsilon_{D})})}{2}\Big)\Gamma\Big(\frac{s -i(t\pm\frac{\pi x}{\log(\varepsilon_{D})})}{2}\Big).
\]
\end{definition}
\noindent When the Maass forms $f$ and $\theta_{\chi_{\psi,n}}$ have the same parity we have $ \gamma(s,f,n) = \gamma(s,f\times \theta_{\chi_{\psi,n}})$.

\begin{lemma}
\label{qinfty bounds}
Let $f$ and $g$ be two Maass forms with respective spectral parameters $t$ and $r$, both real and suppose that $f$ and $g$ have the same parity. Then, for $s\in \mathbb{R}$ fixed, 
\[
\mathfrak{q}_{\infty}(s,f\times g) \asymp_{s} (1+(t+r)^{2})(1+(t-r)^{2})
\]
where $\mathfrak{q}_{\infty}(s,f\times g)$ is the analytic conductor of the L-function $L(s,f\times g)$.
\end{lemma}

\noindent Using Stirling's formula, we can estimate the gamma factors of the Rankin-Selberg  and adjoint square L-function as follows
\begin{lemma}
\label{gamma factor rankin-selberg}
Let $f$ and $g$ be Hecke-Maass cusp forms with real spectral parameters $t$ and $r$ respectively and suppose $f$ and $g$ have the same parity. For $s=1/2$, we have 
\begin{equation}
\gamma\Big(\frac{1}{2},f\times g\Big) \ll  (1+|t+r|^{2})^{-\frac{1}{4}}(1+|t-r|^{2})^{-\frac{1}{4}} (e^{-\frac{\pi}{4}|t + r|})^{2}(e^{-\frac{\pi}{4}|t - r|})^{2}
\end{equation}
where the implied constant is absolute.
\end{lemma}

\begin{lemma}
\label{gamma factor adjoint}
Let $f$ be a Hecke-Maass cusp form of level $q\geq 1$ with spectral parameter $t \in \mathbb{R}$. Then,
\[
    \gamma(1,\Ad f) \gg   e^{-\pi |t| }
\]
where the implied constant is absolute.
\end{lemma}

\noindent Finally, we conclude this section with a theorem by Hoffstein and Lockhart.
\begin{theorem}
\label{adjoint L-function lower bound}
Let $f$ be a Hecke-Maass cusp form of level $q=1$ and Laplace eigenvalue $\lambda$. For any $\varepsilon >0$, we have
\[
L(1,\Ad f) \gg_{\varepsilon} \lambda^{-\varepsilon}.
\]
\end{theorem}
\begin{proof}
See \cite{hoffstein-lockhart}.
\end{proof}

\subsection{Approximate functional equation}
To prove Theorem A we will need to estimate the central values of the L-functions $L(s, f\times \theta_{\chi})$ i.e.\ the values $L(1/2, f\times \theta_{\chi})$ where $\theta_{\chi}$ is the theta function of some Hecke character $\chi$. Since $s=1/2$ lies in the critical strip, we will use the so called approximate functional equation.
For the proofs of the approximate functional equation and Lemma \ref{V bound lemma}, we refer the reader to \cite{IwaniecKowalski} pages 98 and 100 respectively.

\begin{theorem}[Approximate functional equation] 
\label{approximate functional equation}
Let $f,g$ be self-dual Maass cusp forms such that the completed L-function $
\Lambda(s,f\times g)$ is entire. Let $G(u)=e^{u^{2}}$,
keeping the notation of Section \ref{rankin-selberg section}, we have
\begin{align*}
L(1/2, f\times g) &= (1+ \varepsilon(f\times g))\sum_{m\geq 1} \frac{\chi_{f}(m)\chi_{g}(m)}{m}\sum_{n\geq 1}\frac{\lambda_{f}(n)\lambda_{g}(n)}{\sqrt{n}} V_{\frac{1}{2}}(\frac{m^{2}n}{\sqrt{q(f\times g)}}) 
\end{align*}
where
$
V_{s}(y) \coloneqq \frac{1}{2\pi i} \int_{\Re(u) = 3} \frac{\gamma(s+u,f\times g)}{\gamma(s,f\times g)}y^{-u}G(u) \frac{du}{u}
$ and $\varepsilon(f\times g)$ is the root number of $L(s,f\times g)$.
\end{theorem}

\noindent Because of our choice of $G(u)$, $V_{s}(y)$ behaves essentially like a bump function with respect to $y$. More precisely,

\begin{lemma} 
\label{V bound lemma}
Fix $A>0$. Suppose $\Re(s+\kappa_{j}) \geq 3\alpha > 0$ for $j=1,\hdots,4$ where $\kappa_{j}$ are the local parameters of $L(s,f\times g)$ at infinity. Then, the function $V_{s}(y)$ defined in Theorem \ref{approximate functional equation} satisfies
\begin{equation}
\label{eq:V bound}
    V_{s}(y) \ll_{\alpha,A} \Big(1 + \frac{y}{\sqrt{\mathfrak{q}_{\infty}(s,f\times g)}}\Big)^{-A}.
\end{equation}
\end{lemma}

\noindent Using the fast decay of $V_{\frac{1}{2}}(y)$, we can ignore all terms in the approximate functional equation which are large compared to the conductor $\mathfrak{q}_{\infty}(1/2,f\times g)$.

\begin{corollary}
\label{approximate functional equation corollary}
Let $f,g$ be self-dual Maass cusp forms with respective Laplace eigenvalues $t,r \in \mathbb{R}$. Assume that the completed L-function $\Lambda(s,f\times g)$ is entire and the root number $\varepsilon(f\times g) =1$. Then, for any $B, c, \delta >0$ and $M\geq c.\mathfrak{q}_{\infty}(1/2, f\times g)^{\frac{1}{2}+\delta} $ we have
\begin{equation*}
    L(1/2, f\times g) = 2\sum_{m= 1}^{M} \frac{\chi_{f}(m)\chi_{g}(m)}{m}\sum_{n= 1}^{M} \frac{\lambda_{f}(n)\lambda_{g}(n)}{\sqrt{n}} V_{
    \frac{1}{2}}(\frac{m^{2}n}{\sqrt{q(f\times g)}}) + O_{c,B,\delta,q }(\mathfrak{q}_{\infty}(1/2,f\times g)^{-B})
\end{equation*}
where $V_{s}(y)$ is the function described in Lemma \ref{V bound lemma} and the implied constant depends only on $c$, $B$, $\delta$ and $q\coloneq q(f\times g)$.
\end{corollary}

\begin{remark}
\label{remark: V def}
In case $f$ is a Maass form and $\theta_{\chi_{\psi,n}}$ is the theta function associated to the Hecke character $\chi_{\psi,n}$, $V_{s}(y)$ depends only on $n$ and not on $\psi$ (since the gamma factor depends only on $n$). Using the function $\gamma(s,f,x)$, we can define the function $V_{s}(y,x)$ in the obvious way so that $V_{s}(y,n) = V_{s}(y)$ whenever $n$ and $f$ have the same parity. 
\end{remark}

\noindent Combining the approximate functional equation, Lemma \ref{fourier coef theta} and Iwaniec's bound towards the Ramanujan conjecture from \cite{Iwaniec1992}, we have a uniform convexity bound for $L(1/2,f\times \theta_{\chi})$. More precisely,

\begin{proposition}[Convexity bound]
\label{convexity bound uniform}
Let $f$ be a Hecke-Maass cusp form of level 1 and spectral parameter $t$ and let $\chi \coloneqq \chi_{\psi,r}$ be a Hecke character for the quadratic number field $\mathbb{Q}(\sqrt{D})$. Then, for any $\varepsilon>0$ we have
\[
L(1/2,f\times \theta_{\chi_{\psi,r}}) \ll_{\varepsilon,D}  \mathfrak{q}_{\infty}(1/2,f\times \theta_{\chi_{\psi,r}})^{\frac{1}{4}+\varepsilon}.
\]
\end{proposition}

\noindent We conclude this subsection with Proposition \ref{V derivative with respect to n} which roughly states that the derivatives of $V_{1/2}(y,n)$ with respect to $n$ are small. Intuitively, this implies that the approximate functional equation is more or less uniform in $n$ over small intervals. 

\begin{proposition}
\label{V derivative with respect to n}
Let $f$ be a Maass form with spectral parameter $t \geq 1$ and let $V_{s}(y,x)$ be defined as in Remark \ref{remark: V def}. Let $1\leq T \ll_{D} t $ and suppose that $\frac{T}{2}\leq |c_{D}t-x| \leq 2T$ with $c_{D} = \frac{\log(\varepsilon_{D})}{\pi}$. Then, for any $j\geq 1$ we have
\[
\frac{\partial^{j}}{\partial x^{j}} V_{\frac{1}{2}}(y,x) \ll_{\varepsilon,j} y^{-\frac{\varepsilon}{4}} t^{\varepsilon}T^{-j}.
\]
\end{proposition}

\begin{proof}
Let $\delta \coloneqq \varepsilon/4$.
Using Cauchy's theorem we can move the integration in the definition of $V_{\frac{1}{2}}(y,x)$ to the line $\Re(s) = \delta$. Let $F(u, f, x) \coloneqq \frac{\gamma(1/2+u,f,x)}{\gamma(1/2,f,x)}$ so that
$
V_{\frac{1}{2}}(y,x) = I_{1}(y,x) + I_{2}(y,x)
$
where we define  
\[
I_{1}(y,x) \coloneqq \frac{1}{2\pi i} \int_{\substack{\Re(u)=\delta \\ |\Im(u)| \leq \frac{T}{4}}} F(u,f,x) y^{-u}e^{u^{2}} \frac{du}{u}
\text{ and }
I_{2}(y,x) \coloneqq \frac{1}{2\pi i} \int_{\substack{\Re(u)=\delta \\ |\Im(u)| \geq \frac{T}{4}}} F(u,f,x) y^{-u}e^{u^{2}} \frac{du}{u}.
\]
 Using Stirling's formula and estimates of the polygamma functions, we can prove that for $j\geq 0$ and $\Re(u) = \delta$ we have
$
\frac{\partial^{j}}{\partial x^j} F(u,f,x) \ll_{j,\delta,D} |u||F(u,f,x)| \ll_{j,\delta, D} |u| e^{2\pi |\Im(u)|} t^{4\delta}
$
and, if in addition $|\Im(u)| \leq \frac{T}{4}$, we then have
\[
\frac{\partial^{j}}{\partial x^j} F(u,f,x) \ll_{j,\delta,D} |u||F(u,f,x)|T^{-j} \ll_{j,\delta,D} T^{-j} |u|e^{2\pi |\Im(u)|} t^{4\delta} .
\]
As a consequence, $I_{1}(y,x)$ and $I_{2}(y,x)$ both satisfy the conditions for differentiation under the integral sign which leads to the estimates
\begin{align*}
    \frac{\partial^{j}}{\partial x^{j}} I_{1}(y,x) &\ll_{j,\delta,D} y^{-\delta} t^{4\delta} T^{-j} \int_{\substack{\Re(u)=\delta \\ |\Im(u)| \leq \frac{T}{4}}} e^{2\pi |\Im(u)|}e^{\delta^{2}-\Im(u)^{2}} \, du \ll_{j,\delta,D} y^{-\delta} t^{4\delta} T^{-j}
\end{align*}
and
\begin{align*}
     \frac{\partial ^{j}}{\partial x^{j}} I_{2}(y,x)
     &\ll_{j,\delta,D} y^{-\delta} t^{4\delta} e^{-\frac{T}{2}(T/8-\pi)} \int_{0}^{\infty} e^{2\pi w} e^{-w^{2}} \,dw \ll_{j.\delta,D} y^{-\delta} t^{4\delta} T^{-j}.
\end{align*}

\end{proof}

\subsection{Dyadic partition of unity}
Partitions of unity allow us to break down sums over $\mathbb{Z}$ into sums over smaller intervals which are sometimes easier to estimate (see \cite{blomer-michel-kowalski}, Section 2.8).
\begin{proposition}
\label{dyadic partition of unity}
There exists a smooth function $W(x)$ supported on $[\frac{1}{2},2]$ and such that
\[
\sum_{k\geq 0} W\Big(\frac{x}{2^{k}}\Big) = 1
\]
for any $x\geq 1$.
\end{proposition}

\subsection{Lipschitz principle}
Finally, in the proof of Theorem A, we will need to estimate the number of lattice points contained inside a parallelogram in $\mathbb{R}^{2}$. The Lipschitz principle gives an upper bound with respect to the area of the parallelogram and the length of its boundary. 
\begin{proposition}
\label{lipschitz principle}
Let $R$ be a closed and bounded region in $\mathbb{R}^{2}$, $L$ the length of the boundary of $R$ and $A$ the area of $R$. Then,
\[
|R\cap \mathbb{Z}^{2}| \ll L + A+ 1.
\]
\end{proposition}
\begin{proof}
This follows from the Lipschitz principle  \cite{lipschitz-principle} or Jarnik's inequality \cite{Jarnik-inequality}.
\end{proof}
\begin{corollary}
\label{lattice points parallelogram bounds}
Let $\Lambda$ be a lattice in $\mathbb{R}^{2}$ with (ordered) basis $\mathcal{B} = \{ \alpha_{1},\alpha_{2} \}$ and let $v_{1},v_{2} \in \mathbb{R}^{2}$ be two linearly independent vectors.
If we let $\mathcal{P}$ be the parallelogram
$
\mathcal{P} \coloneqq \{ av_{1} + bv_{2} \mid a,b \in [0,1] \} 
$
then
\[
| \mathcal{P} \cap \Lambda | \ll_{\mathcal{B}} (||v_{1}|| + ||v_{2}||) + ||v_{1}||.||v_{2}|| +1
\]
where the implied constant depends only on the basis $\mathcal{B}$ of $\Lambda$.
\end{corollary}

\begin{proof}
Let us first assume that $\Lambda = \mathbb{Z}^{2}$. Then, by Proposition \ref{lipschitz principle} we know that
$
|\mathcal{P}\cap \mathbb{Z}^{2}| \ll L + A + 1 
$
where $L = 2(||v_{1}|| + ||v_{2}||)$ is the perimeter of $\mathcal{P}$ and $A = | v_{1} \times v_{2} |\leq ||v_{1}||.||v_{2}||$. The result follows immediately in this case.
\newline
\newline
Now let $\Lambda$ be any lattice and $\mathcal{B}$ be a fixed ordered basis. We reduce to the case above using an appropriate linear transformation. More precisely, let $T \colon \mathbb{R}^{2} \to \mathbb{R}^{2}$ be the linear isomorphism defined by
$
T(\alpha_{1}) = (1,0) \text{ and } T(\alpha_{2})= (0,1).
$
Note that $T$ only depends on our (ordered) basis $\mathcal{B}$, and therefore, so does its operator norm. In other words, we have
$
||Tv_{1}|| \ll_{\mathcal{B}} ||v_{1}|| \text{ and } ||Tv_{2}||\ll_{\mathcal{B}} ||v_{2}||.
$
In particular, since $T\mathcal{P}$ is the parallelogram
$
T\mathcal{P} = \{ a(Tv_{1}) + b(Tv_{2}) \mid a, b \in [0,1] \}
$,
we obviously have the estimate
\[
|T\mathcal{P} \cap \mathbb{Z}^{2} | \ll_{\mathcal{B}} \Big(||v_{1}||+||v_{2}||\Big) + ||v_{1}||.||v_{2}|| + 1.
\]
Finally, $T$ is an isomorphism and $T\Lambda = \mathbb{Z}^{2}$, hence
$|\mathcal{P} \cap \Lambda | = | T\mathcal{P} \cap \mathbb{Z}^{2} |$ which concludes the proof.
\end{proof}

\section{Proof of Theorem A}
\noindent We are finally ready to prove the main result. In what follows, we will let $f \in L^{2}(\Sl2(\mathbb{Z})\setminus\mathbb{H})$ be a Hecke-Maass cusp with spectral parameter $t\geq 2$.

\subsection{Restriction norm formula}
\label{sec: restriction norm formula}
We start by using representation theory of compact abelian groups (or equivalently, Fourier theory of compact abelian groups) to derive an explicit formula for the restriction norm $||f|_{\mathcal{C}_{D}}||_{2}^{2}$. 
\newline
\newline
Keeping the notation defined in Section \ref{sec: background material}, we consider a closed geodesic $\gamma$ in $\Sl2(\mathbb{Z})\setminus\mathbb{H}$ of length $\log(\varepsilon_{D}^{2})$ and a lift $\tilde{\gamma}$ of $\gamma$ in $\mathbb{H}$. Then $\tilde{\gamma}$ is a semi-circle perpendicular to the real axis with endpoints $\omega, \omega^{*} \in K \subseteq \mathbb{R}$ with $\omega > \omega^{*}$ (as usual $\omega^{*}$ denotes the Galois conjugate of $\omega$). Consider the matrix
\[
\kappa_{\gamma} \coloneqq \frac{1}{\sqrt{\omega - \omega^{*}}} \begin{pmatrix}
1& -\omega \\
1 &-\omega^{*} 
\end{pmatrix} \in \Sl2(\mathbb{R}).
\]
One easily checks that $\kappa_{\gamma}.\omega = 0$ and $\kappa_{\gamma}.\omega^{*}= \infty$, and therefore, $\kappa_{\gamma} \tilde{\gamma} = \gamma_{0}$ where $\gamma_{0} = i\mathbb{R}^{+}$ is the imaginary axis in the upper half plane. Since the hyperbolic measure $ds$ is invariant under the action of $\Sl2(\mathbb{R})$, for any $g\colon \tilde{\gamma} \to \mathbb{C}$ we have
$
\int_{\tilde{\gamma}} g(z) \, ds = \int_{0}^{\infty} g(\kappa_{\gamma}^{-1}.iy) \, \frac{dy}{y}.
$
For any $g\colon \gamma \to \mathbb{C}$, this leads to
\begin{equation}
\label{eq:parametrization imaginary axis}
    \int_{\gamma} g(z) \, ds = \int_{1}^{\varepsilon_{D}^{2}} g(\kappa_{\gamma}^{-1}.iy) \, \frac{dy}{y}.
\end{equation}
We can therefore express integration over $\gamma$ using the Lebesgue integral over $\mathbb{R}$. 
Introducing the change of variables $x\coloneqq \log(\varepsilon_{D}^{2})^{-1}\log(y)$ or equivalently, $y=e^{\log(\varepsilon_{D}^{2})x}$, (\ref{eq:parametrization imaginary axis}) becomes
\begin{equation}
\label{eq:parametrization circle}
    \int_{\gamma} g(z) \, ds =  \log(\varepsilon_{D}^{2}) \int_{0}^{1} g(\kappa_{\gamma}^{-1}. ie^{\log(\varepsilon_{D}^{2})x}) \, dx.
\end{equation}
By abuse of notation, we will denote the set $\{ iy \mid 1\leq y < \varepsilon_{D}^{2} \}$ by $\kappa_{\gamma} \gamma$. Geometrically, if we define a parametrization of $\kappa_{\gamma} \gamma$ as follows
\begin{align*}
    \Phi_{0} \colon [0,1) &\to \kappa_{\gamma} \gamma \\
     x &\mapsto ie^{\log(\varepsilon_{D}^{2})x}
\end{align*}
then, the right hand side of (\ref{eq:parametrization circle}) is the integral of the pullback $\Phi_{0}^{*}g(\kappa_{\gamma}^{-1} \cdot)$ along $[0,1)$. As a consequence, after canonically identifying $[0,1)$ to $\mathbb{R}/\mathbb{Z}$ (which we endow with the Haar probability), we have a homeomorphism between $\kappa_{\gamma} \gamma$ and a compact abelian group which is compatible with the measures on both spaces up to scaling. In particular, it preserves the corresponding $L^{2}$-inner products in the sense that 
\begin{equation}
    \langle g_{1},g_{2} \rangle = \log(\varepsilon_{D}^{2}) \langle \Phi_{0}^{*}g_{1}, \Phi_{0}^{*}g_{2} \rangle
\end{equation}
for every $g_{1},g_{2} \colon \kappa_{\gamma} \gamma \to \mathbb{C}$.
Of course, the inner product on the left hand side takes place in $L^{2}(\kappa_{\gamma} \gamma)$ while the inner product on the right hand side takes place in $L^{2}(\mathbb{R}/\mathbb{Z})$.
\newline
\newline
Using Theorem \ref{closed geodesics- class group} we can therefore parametrize the union of closed geodesics $\mathcal{C}_{D}$ using
\begin{align*}
    \Xi \colon \mathbb{R}/\mathbb{Z}\times H^{+}(K) &\to \mathcal{C}_{D} \\
    ([x], [\mathfrak{a}]) &\mapsto \kappa_{\gamma_{\mathfrak{a}}}^{-1}ie^{\log(\varepsilon_{D}^{2})x}
\end{align*} 
where $\gamma_{\mathfrak{a}}$ is the closed geodesic corresponding to the class group element $[\mathfrak{a}]$. When we endow the finite group $H^{+}(K)$ with the discrete topology, $\Xi$  is a homeomorphism which is compatible with integration on both spaces. Indeed, for any integrable $g\colon \mathcal{C}_{D} \to \mathbb{C}$ we have
\begin{align*}
    \int_{\mathcal{C}_{D}} g(z)\, ds &= \sum_{[\mathfrak{a}] \in H^{+}(K)} \int_{\gamma_{[\mathfrak{a}]}} g(z) \,ds \\
    &= \log(\varepsilon_{D}^{2})|H^{+}(K)| \int_{\mathbb{R}/\mathbb{Z} \times H^{+}(K)} \Xi^{*}g(z,[\mathfrak{a}]) \, d(z,[\mathfrak{a}])
\end{align*}
using the Fubini theorem and (\ref{eq:parametrization circle}). In terms of $L^{2}$ inner-products this translates to
\begin{equation}
\label{eq: inner products relation}
     \langle g_{1},g_{2} \rangle = |\mathcal{C}_{D}| \langle \Xi^{*}g_{1}, \Xi^{*}g_{2} \rangle
\end{equation}
for any $g_{1},g_{2} \colon \mathcal{C}_{D} \to \mathbb{C}$. We can therefore use Plancherel to express the $L^{2}$ norm of any $g \in L^{2}(\mathcal{C}_{D})$.
\begin{theorem}
For any $g \in L^{2}(\mathcal{C}_{D})$, we have
\begin{equation}
\label{eq: L2 norm explicit}
    ||g||_{2}^{2} = \frac{1}{ |\mathcal{C}_{D}|}\sum_{\psi\in \widehat{H^{+}(K)}} \sum_{n \in \mathbb{Z}} I(g,\chi_{\psi,n}).
\end{equation}
where $\chi_{\psi,n}$ is the Hecke character obtained from $\psi \in \widehat{H^{+}(K)}$, $n \in \mathbb{Z}$ and
\begin{equation}
    I(g,\chi_{\psi, n}) \coloneqq \Big| \sum_{[\mathfrak{a}] \in H^{+}(K)} \psi([\mathfrak{a}])^{-1} \int_{1}^{\varepsilon_{D}^{2}}g(\kappa_{\gamma_{[\mathfrak{a}]}}^{-1}iy) y^{-\frac{n\pi i }{\log(\varepsilon_{D})}} \, \frac{dy}{y} \Big|^{2}.
\end{equation}
\end{theorem}

\begin{proof}
Using (\ref{eq: inner products relation}) we have
$
||g||_{2}^{2} = |\mathcal{C}_{D}| \langle \Xi^{*}g, \Xi^{*}g \rangle
$
where the inner product on the right hand side takes place in  $L^{2}(\mathbb{R}/\mathbb{Z} \times H^{+}(K))$.
Since $G \coloneqq \mathbb{R}/\mathbb{Z}\times H^{+}(K)$ is a compact abelian Lie group equipped with the Haar probability, the characters of $G$ form an orthonormal basis for $L^{2}(G)$.
The dual group of $G$ is equal to
\[
\hat{G}= \{ \chi_{\psi,n} \colon (x,[\mathfrak{a}]) \mapsto e^{2\pi i nx }\psi([\mathfrak{a}]) \mid n \in \mathbb{Z}, \psi \in \widehat{H^{+}(K)} \}.
\]
Therefore by Plancherel, 
\begin{align*}
    ||g||_{2}^{2} &= |\mathcal{C}_{D}|  \sum_{\psi \in \widehat{H^{+}(K)}}\sum_{n \in \mathbb{Z}} \Big| \langle \Xi^{*} g, \chi_{\psi,n} \rangle \Big|^{2}
\end{align*}
where the inner products on the right hand side take place in $L^{2}(G)$.
Finally, for any $\psi$ and $n$, we have
\begin{align*}
    \langle \Xi^{*} g, \chi_{\psi,n} \rangle &= |H^{+}(K)|^{-1} \sum_{[\mathfrak{a}]\in H^{+}(K)} \int_{0}^{1} g(\kappa_{\gamma_{[\mathfrak{a}}]}^{-1}ie^{\log(\varepsilon_{D}^{2})x}) e^{-2\pi i n x} \overline{\psi}([\mathfrak{a}]) \, dx \\
    &= |\mathcal{C}_{D}|^{-1} \sum_{[\mathfrak{a}]\in H^{+}(K)}  \psi([\mathfrak{a}])^{-1} \int_{0}^{\varepsilon_{D}^{2}} g(\kappa_{\gamma_{[\mathfrak{a}}]}^{-1}iy) y^{\frac{-\pi i n }{ \log(\varepsilon_{D})}}  \, \frac{dy}{y}
\end{align*}
after applying the change of variables $y \coloneqq e^{\log(\varepsilon_{D}^{2})x}$. Combining everything, we recover (\ref{eq: L2 norm explicit}).
\end{proof}

\noindent Whenever $f \in L^{2}(\Sl2(\mathbb{Z})\setminus \mathbb{H})$ is a Hecke-Maass cusp form, it follows from Waldspurger's formula (Theorem 6.3.1 in \cite{popa} or more generally Theorem 4.1 in \cite{martin-whitehouse}) that
up to some positive constant, 
\begin{equation*}
\label{eq: completed L function}
    I(f,\chi) \asymp \frac{1}{\Lambda(1,\Ad f)} \Lambda(1/2, f\times \theta_{\chi})
\end{equation*}
where $\Lambda(1,\Ad f)$ is the completed adjoint square L-function of $f$ and $\Lambda(1/2, f\times \theta_{\chi})$ is the completed Rankin-Selberg L-function. In this case, the restriction norm formula becomes
\begin{equation*}
    ||f|_{\mathcal{C}_{D}}||_{2}^{2} \asymp   \frac{1}{ |\mathcal{C}_{D}|}\sum_{\psi\in \widehat{H^{+}(K)}} \sum_{n \in \mathbb{Z}}\frac{1}{\Lambda(1,\Ad f)} \Lambda(1/2, f\times \theta_{\chi_{\psi,n}}).
\end{equation*}
\begin{remark}
\label{remark: parity}
Whenever $f$ and $n$ have different parities (i.e.\ $f$ is an even Maass form and $n$ is an odd number or $f$ is an odd Maass form and $n$ is an even number), the integral $I(f,\chi_{\psi,n})$ vanishes by symmetry. Therefore, we are actually taking the sum over all $n$ with the same parity as $f$. 
\end{remark}

\noindent  By definition, for any Hecke character $\chi_{\psi,n}$ of $K$ we have
\[
\frac{\Lambda(1/2, f\times \theta_{\chi_{\psi,n}})}{\Lambda(1,\Ad f)} = \frac{D^{\frac{1}{2}}\gamma(1/2,f\times \theta_{\chi_{\psi,n}})L(1/2,f\times \theta_{\chi_{\psi,n}})}{\gamma(1,\Ad f) L(1, \Ad f)}.
\]
If we define
\[
G(n) \coloneqq \frac{\gamma(1/2,f\times \theta_{\chi_{\psi,n}})}{\gamma(1,\Ad f)}
\]
to be the corresponding quotient of gamma factors, then, by Theorem \ref{adjoint L-function lower bound} we have the estimate
\begin{equation}
\label{eq: restriction norm formula}
    \frac{1}{ |\mathcal{C}_{D}|}\sum_{\psi\in \widehat{H^{+}(K)}} \sum_{n \in \mathbb{Z}}\frac{1}{\Lambda(1,\Ad f)} \Lambda(1/2, f\times \theta_{\chi_{\phi,n}}) \ll_{\varepsilon} t^{\varepsilon}  \frac{D^{\frac{1}{2}}}{ |\mathcal{C}_{D}|}\sum_{\psi\in \widehat{H^{+}(K)}} \sum_{n \in \mathbb{Z}} G(n) L(1/2,f\times \theta_{\chi_{\psi,n}})
\end{equation}
for any $\varepsilon>0$.

\begin{remark}
\label{remark: positivity}
For any $z\in \mathbb{C}$, $\Gamma(z)\Gamma(\overline{z}) = |\Gamma(z)|^{2}\geq 0$. If follows from the definition of $G(n)$ that
$
G(n) \geq 0
$
for all $n\in \mathbb{Z}$. As a consequence of Waldspurger's formula we must also have 
$ L(1/2,f\times \theta_{\chi_{\psi,n}})\geq 0
$ for any $\psi \in \widehat{H^{+}(K)}$ and $n\in\mathbb{Z}$. 
\end{remark}

\subsection{Estimation of the gamma factors}
 To estimate the restriction norm, we start by applying Stirling's formula to $G(n)$. As a consequence, we will deduce that the right hand side of (\ref{eq: restriction norm formula}) is essentially a finite sum.
\begin{lemma}
\label{G(n) estimates}
For any $n \in \mathbb{Z}$, let $r\coloneqq \frac{\pi n}{\log(\varepsilon_{D})}$. Then,
\begin{equation}
    \label{eq: estimate gamma all n}
    G(n) \ll 
(1+|t+r|^{2})^{-\frac{1}{4}}(1+|t-r|^{2})^{-\frac{1}{4}} e^{-\pi(\max\{|r|,|t|\}-|t|)}
\end{equation}
absolutely.
Moreover, if $|n|\geq 2\frac{\log(\varepsilon_{D})}{\pi}t$, then for any $\varepsilon>0$ we have
\begin{equation}
    \label{eq: estimate gamma large n}
        G(n) \ll_{\varepsilon}
(1+|t+r|^{2})^{-\frac{1}{4}-\varepsilon}(1+|t-r|^{2})^{-\frac{1}{4}-\varepsilon} e^{-\frac{\pi}{4}|r|}.
\end{equation}
\end{lemma}

\begin{proof}
By Remark \ref{remark: parity}, we may assume $n$ and $f$ have the same parity.
Using Lemmas \ref{gamma factor rankin-selberg} and \ref{gamma factor adjoint}, we have
\begin{equation}
\label{eq: gamma error term 1}
    G(n) \ll \frac{(1+|t+r|^{2})^{-\frac{1}{4}}(1+|t-r|^{2})^{-\frac{1}{4}} (e^{-\frac{\pi}{4}|t + r|})^{2}(e^{-\frac{\pi}{4}|t - r|})^{2}}{e^{-\pi|t|}}
\end{equation}
absolutely. 
Using the well known equality $|t+r|+|t-r| = 2\max\{|t|,|r|\}$, we get (\ref{eq: estimate gamma all n}).
To prove (\ref{eq: estimate gamma large n}), we assume that $|r|\geq 2|t|$. In this case, $\max\{ |t|, |r| \} = |r|$ and therefore
\begin{align*}
 (e^{-\frac{\pi}{4}|t + r|})^{2}(e^{-\frac{\pi}{4}|t - r|})^{2} \leq e^{-\pi |t| } e^{-\frac{\pi|r| }{4}}e^{-\frac{\pi}{8}|t + r|}e^{-\frac{\pi}{8}|t - r|} \ll_{\varepsilon} e^{-\pi |t| } e^{-\frac{\pi|r| }{4}}(1+|t+r|^{2})^{-\varepsilon}(1+|t-r|^{2})^{-\varepsilon}.
\end{align*}
\end{proof}

\noindent Let $c_{D}\coloneqq \frac{\log(\varepsilon_{D})}{\pi}$. We can split the sum we want to estimate as follows
\begin{equation}
   \frac{t^{\varepsilon}D^{\frac{1}{2}}}{|\mathcal{C}_{D}|}\sum_{|n|\leq c_{D}t}\sum_{\psi\in \widehat{H^{+}(K)}}G(n)L(1/2,f\times \theta_{\chi_{\psi,n}}) + \frac{t^{\varepsilon}D^{\frac{1}{2}}}{|\mathcal{C}_{D}|} \sum_{|n| >c_{D}t } \sum_{\psi\in \widehat{H^{+}(K)}} G(n)L(1/2,f\times \theta_{\chi_{\psi,n}}).
\end{equation}
Then, by Lemma \ref{G(n) estimates}, we expect the second part to be a negligible error term. We start by proving that it is indeed the case.

\begin{proposition}
\label{error term}
For any $\varepsilon>0$, we have
\begin{equation*}
  \sum_{|n| >c_{D}t } \sum_{\psi\in \widehat{H^{+}(K)}} G(n)L(1/2,f\times \theta_{\chi_{\psi,n}}) \ll_{D,\varepsilon} t^{\varepsilon}. \end{equation*} 
\end{proposition}

\begin{proof}
Using  Proposition \ref{convexity bound uniform} and Lemma \ref{qinfty bounds}  we estimate $L(1/2,f\times \theta_{\chi_{\psi,n}})$ uniformly in $\chi_{\psi,n}$. We then have
\[
L(1/2,f\times \theta_{\chi_{\psi,n}}) \ll_{\varepsilon}  (1+|t+\frac{n}{c_{D}}|^{2})^{\frac{1}{4}+\varepsilon} (1+|t-\frac{n}{c_{D}}|^{2})^{\frac{1}{4}+\varepsilon}.
\]
Combining this estimate with (\ref{eq: estimate gamma large n}) and taking the sum over all $n\geq 2c_{D}t$ we see that
\begin{align*}
    \sum_{|n|\geq 2c_{D}t } G(n)L(1/2,f\times \theta_{\chi_{\psi,n}}) \ll_{D,\varepsilon} \sum_{n \geq 2c_{D}t} e^{-\frac{\pi}{4c_{D}}n} \ll_{D,\varepsilon} t^{\varepsilon}.
\end{align*}
On the other hand, by (\ref{eq: estimate gamma all n}) for $c_{D}t < |n| < 2c_{D}t$ we have
\[
G(n)L(1/2,f\times \theta_{\chi_{\psi,n}}) \ll_{\varepsilon} (1+|t+\frac{n}{c_{D}}|^{2})^{\varepsilon} (1+|t-\frac{n}{c_{D}}|^{2})^{\varepsilon} e^{-\pi(\frac{|n|}{c_{D}}-|t|)}.
\]
But in this case, $|t+\frac{n}{c_{D}}| \leq 3|t|$ and $|t-\frac{n}{c_{D}}| \leq |t| $ so that
\begin{align*}
    \sum_{c_{D}t<| n|  < 2c_{D}t} G(n)L(1/2,f\times \theta_{\chi_{\psi,n}}) &\ll_{\varepsilon} t^{4\varepsilon} \sum_{n=\lceil c_{D}t\rceil }^{\infty} e^{-\pi(\frac{n}{c_{D}}-|t|)} \ll_{\varepsilon,D} t^{4\varepsilon}.
\end{align*}
Finally, note that both estimates are independent of $\psi$. Since $|\widehat{H^{+}(K)}|$ is a finite constant depending only on $D$, we have proved the proposition.
\end{proof}

\begin{remark}
Applying Lemma \ref{G(n) estimates} and the convexity bound to the main term leads to the bound
\[
||f|_{\mathcal{C}_{D}}||_{2} \ll_{\varepsilon,D} t^{1+\varepsilon}
\]
which is significantly worse than the bound from Theorem A. It is also interesting to note that the Lindel\"of hypothesis automatically implies the bound
$
||f|_{\mathcal{C}_{D}}||_{2} \ll_{\varepsilon,D} t^{\varepsilon}.
$
\end{remark}
\subsection{First smooth partition of unity}

\noindent We now move on to the main term. We start by introducing a smooth partition of unity centered about $c_{D}t$. More precisely, let $\tilde{U}(x)$ be a smooth bump function which is equal to 1 on $[0,\infty)$ and supported on $[-1/2, \infty)$. We can define the smooth bump function 
\[
U(x)\coloneqq \tilde{U}\Big( \frac{x}{c_{D}t} \Big)
\]
which is equal to 1 on $[0,\infty)$ and is supported on $[-c_{D}t/2,\infty)$. It is easy to see that for any $x \in \mathbb{R}$, we have the inequality
$
 U(x) + U(-x) \geq 1.
$
But for any $n\in \mathbb{Z}$ and $\psi \in \widehat{H^{+}(K)}$ we know that $G(n)L(1/2,f\times \theta_{\chi_{\psi,n}}) \geq 0$, $G(n)=G(-n)$ and $\theta_{\chi_{\psi,n}} = \theta_{\chi_{\psi,-n}}$, hence,
\[
  \sum_{|n|\leq c_{D}t} G(n)L(1/2,f\times \theta_{\chi_{\psi,n}}) \leq 2 \sum_{|n|\leq c_{D}t}  U(n)G(n)L(1/2,f\times \theta_{\chi_{\psi,n}}).
\]

\noindent Now let $W(x)$ be the function obtained using Proposition \ref{dyadic partition of unity}. For each $k\geq 0$ we define the function
\[
W_{k} (x) \coloneqq W\Big(\frac{c_{D}t-x}{2^{k}}\Big)
\]

\noindent By construction, for any $k\geq 0$, $W_{k}$ is supported on $2^{k-1} \leq c_{D}t-x \leq 2^{k+1} $ and
\begin{align*}
    \sum_{k\geq 0} W_{k}(x)= \begin{cases} 
    1 & \text{for } c_{D}t-x \geq 1 \\
    0 &\text{for } x \geq c_{D}t.
    \end{cases}
\end{align*}
Using Stirling's formula and the convexity bound (Proposition \ref{convexity bound uniform}), we can show that for any $\varepsilon> 0$ and $\psi \in \widehat{H^{+}(K)}$, we have
\begin{equation}
\label{eq: partition of unity, positives}
\sum_{|n|\leq c_{D}t } U(n) G(n) L(1/2,f\times \theta_{\chi_{\psi,n}})  =  \sum_{k\geq 0} \sum_{n\in \mathbb{Z}} W_{k}(n) U(n)G(n)L(1/2,f\times\theta_{\chi_{\psi,n}})+ O_{D,\varepsilon} (t^{\varepsilon}).
\end{equation}

\noindent Let $k_{D} \coloneqq \log_{2}(\frac{3c_{D}t}{2})+1$. Obviously, for $t\geq 2$ we have $k_{D} \ll_{D} \log(t)$. Moreover, since $U(n)W_{k}(n) = 0$ for all $n\in \mathbb{Z}$ whenever $2^{k-1} \geq \frac{3c_{D}t}{2}$, all terms in (\ref{eq: partition of unity, positives}) with $k\geq k_{D}$ may be ignored. Therefore
\begin{align*}
     ||f|_{\mathcal{C}_{D}}||_{2}^{2} \ll_{D,\varepsilon} t^{\varepsilon} \sup_{0\leq k \leq k_{D}} \sum_{\psi \in \widehat{H^{+}(K)}}  \sum_{n\in \mathbb{Z}} W_{k}( n) U( n)G(n)L(1/2,f\times\theta_{\chi_{\psi,n}})+ t^{\varepsilon}
\end{align*}
for any $\varepsilon >0$. 
\newline
\newline
In other words, introducing a partition of unity allowed us to replace the sum in the restriction norm formula by sums over smaller intervals. On each such interval, the conductor $\mathfrak{q}_{\infty}(1/2, f\times \theta_{\chi_{\psi,n}})$ (and therefore the gamma factor $G(n)$) is almost constant with respect to $n$. More precisely, 
\begin{lemma} 
Let $0\leq k \leq k_{D}$ and write $T\coloneqq 2^{k}$. Then,
\begin{equation}
\label{eq: G(n) estimate T}
    W_{k}(n)G(n) \ll_{D} (Tt)^{-\frac{1}{2}}
\end{equation}
for any $n\in \mathbb{Z}$ and
\begin{equation}
\label{eq: qinfty estimate T}
    \mathfrak{q}_{\infty}(s, f\times \theta_{\chi_{\psi,n}}) \asymp_{D} (Tt)^{2}
\end{equation}
for any $n\geq -\frac{c_{D}t}{2}$ satisfying $T/2 \leq c_{D}t - n \leq 2T$. In both cases, the implied constants depend only on $D$. 
\end{lemma} 
\begin{proof}
For any $n\geq -\frac{c_{D}t}{2}$ such that $T/2 \leq c_{D}t - n \leq 2T$, we have $\Big|t+\frac{n}{c_{D}}\Big| \asymp t$. 
Since $W_{k}$ is supported on $T/2 \leq c_{D}t-n \leq 2T$, the results follow immediately from Lemmas \ref{G(n) estimates} and \ref{qinfty bounds} respectively. 
\end{proof}

\noindent Hence, as an immediate consequence of (\ref{eq: G(n) estimate T}), for any $T\coloneqq 2^{k}$ satisfying the conditions of the lemma above 
\[\sum_{\psi\in \widehat{H^{+}(K)}}  \sum_{n\in \mathbb{Z}} W_{k}(n)G(n) U(n)L(1/2,f\times\theta_{\chi_{\psi,n}})  \ll_{D} (Tt)^{-\frac{1}{2}} \sum_{\psi\in \widehat{H^{+}(K)}}  \sum_{n\in \mathbb{Z}} W_{k}(n) U(n)L(1/2,f\times\theta_{\chi_{\psi,n}}).
\]
In what follows, we will prove that for any $\varepsilon>0$, $0\leq k \leq k_{D}$ and $T\coloneqq 2^{k}$ we have
\begin{equation}
\label{eq: smooth parition final}
(Tt)^{-\frac{1}{2}}\sum_{\psi \in \widehat{H^{+}(K)}}  \sum_{n\in \mathbb{Z}} W_{k}( n) U( n)L(1/2,f\times\theta_{\chi_{\psi,n}}) \ll_{\varepsilon,D} t^{\varepsilon+ 2\theta}
\end{equation}
with $\theta$ any bound towards the Ramanujan conjecture.

\subsection{Character orthogonality}
 \noindent We now fix a value of $0\leq k \leq k_{D}$ and write $T\coloneqq 2^{k}$. The first step for proving (\ref{eq: smooth parition final})
consists of writing out the approximate functional equation for $L(1/2, f\times \theta_{\chi_{\psi,n}})$ and applying character orthogonality for the finite group $\widehat{H^{+}(K)}$ to get some cancellations. 
\begin{proposition}
For any $k$ fixed, $\delta >0$ and $B\geq 1$ we have 
\begin{multline}
\label{eq: character orthogonality}
   (Tt)^{-\frac{1}{2}} \sum_{\psi\in \widehat{H^{+}(K)}}\sum_{n\in \mathbb{Z}} W_{k}(n) U(n)L(1/2,f\times\theta_{\chi_{\psi,n}}) =\\ 2 |H^{+}(K)| \sum_{m= 1}^{(Tt)^{1+\delta}} \frac{\chi_{D}(m)}{m}   \sum_{\substack {\alpha \in \mathcal{F}_{D} \\ 1\leq |N(\alpha)| \leq (Tt)^{1+\delta}}}  \frac{\lambda_{f}(\alpha)}{\sqrt{|N(\alpha)|}}\Pi_{\alpha}(m) + O_{B,\delta,D}(t^{-B})
\end{multline}
where $\mathcal{F}_{D}$ is the fundamental domain of the action of $U^{+}(K)$ on $\mathcal{O}_{K}^{+}$ from Proposition \ref{fundamental domain},
\[
\Pi_{\alpha}(m) \coloneqq (Tt)^{-\frac{1}{2}}\sum_{n \in \mathbb{Z}}  W_{k}(n)U(n) V_{\frac{1}{2}}\Big(\frac{m^{2}|N(\alpha)|}{D},n\Big)\Big| \frac{\alpha}{\alpha^{*}}\Big|^{\frac{\pi i n}{\log(\varepsilon_{D})}}
\]
and
$
\lambda_{f}(\alpha) \coloneqq \lambda_{f}(|N(\alpha)|)
$
for any $\alpha \in \mathcal{F}_{D}$, $m\geq 1$ integer. 
\end{proposition}

\begin{proof}
Fix $n\in \mathbb{Z}$ such that $W_{k}(n)U(n) \neq 0$. Then, $n$ lies in the support of $W_{k}$ or equivalently, $T/2 \leq c_{D}t - n \leq 2T $. By (\ref{eq: qinfty estimate T}), $\mathfrak{q}_{\infty}(1/2,f\times \theta_{\chi_{\psi,n}}) \ll_{D} (Tt)^{2}$ and therefore 
we can apply the approximate functional equation (Corollary \ref{approximate functional equation corollary}) with $M=(Tt)^{1+\delta}$. By definition of  $\theta_{\chi_{\psi,n}}$ we then have (up to some constant)
\begin{multline*}
     \sum_{\psi\in \widehat{H^{+}(K)}}L(1/2,f\times\theta_{\chi_{\psi,n}}) =\\  2\sum_{m= 1}^{M} \frac{\chi_{D}(m)}{m}\sum_{r= 1}^{M} \frac{\lambda_{f}(r)}{\sqrt{r}} \sum_{\substack{\mathfrak{a}=\alpha J_{1}^{d_1}...J_{h}^{d_h} \\ \mathbb{N}(\mathfrak{a}) = r}} \Omega_{\mathfrak{a}} \Big|\frac{\alpha}{\alpha^{*}} \Big|^{\frac{\pi i n}{\log(\varepsilon_{D})}}V_{
    \frac{1}{2}}\Big(\frac{m^{2}r}{D},n\Big)  + O_{c,B,\delta,D }((Tt)^{-B}).
\end{multline*}where
\begin{align*}
    \Omega_{\mathfrak{a}} \coloneqq  \sum_{\psi\in \widehat{H^{+}(K)}} \psi([\mathfrak{a}]) =  \begin{cases}
    |H^{+}(K)| &\text{ if } \mathfrak{a} = (x) \text{ for some } x \in \mathcal{F}_{D} \\
    0 &\text{ else }
    \end{cases}
\end{align*}
\noindent by character orthogonality.
\noindent  We conclude the proof by taking the sum over $n\in \mathbb{Z}$. Since the support of $W_{k}(n)U(n)$ is contained in $[c_{D}t - 2T, c_{D}t- T/2]$, the sum consists of precisely $3T/2$ terms which leads to the error term $O_{B,\varepsilon,D}(T^{-B+1}t^{-B}) = O_{B,\varepsilon,D}(t^{-B})$. On the other hand, we obtain the main term of (\ref{eq: character orthogonality}) by permuting the order of summation and rearranging the terms.
\end{proof}

\noindent To estimate the summation on the right hand side of (\ref{eq: character orthogonality}) we introduce another smooth dyadic partition of unity. Let $a_{T}$ be the smallest integer with $a_{T} > (tT)^{1+\delta}$ and let $U_{T}$ be a smooth bump function supported on $[0,a_{T} ]$ and equal to 1 on $[1,(Tt)^{1+\delta}]$. By abuse of notation, we write
\[
U_{T}(\alpha) \coloneqq U_{T}(|N(\alpha)|)
\]
for all $\alpha \in \mathcal{F}_{D}$.
Let $W(x)$ be the function from Proposition \ref{dyadic partition of unity}, and define 
\[
W_{a}^{'} (\alpha) \coloneqq W\Big(\frac{|N(\alpha)|}{2^{a}}\Big)
\]
for $a\geq 0$ and $\alpha \in \mathcal{F}_{D}$. 
Then, for all $1\leq m \leq (Tt)^{1+\delta}$ we have
\[
 \frac{\chi_{D}(m)}{m}   \sum_{\substack {\alpha \in \mathcal{F}_{D} \\ 1\leq |N(\alpha)| \leq (Tt)^{1+\delta}}}  \frac{\lambda_{f}(\alpha)}{\sqrt{|N(\alpha)|}}\Pi_{\alpha}(m) = \sum_{a\geq 0}  \frac{\chi_{D}(m)}{m}   \sum_{\alpha \in \mathcal{F}_{D}}  \frac{\lambda_{f}(\alpha)}{\sqrt{|N(\alpha)|}}\Pi_{\alpha}(m)W_{a}^{'} (\alpha)U_{T}(\alpha).
\]
\newline 
\newline 
Again, whenever $2^{a-1}\geq a_{T} $, $W^{'}_{a}(\alpha)U_{T}(\alpha) =0 $ for all $\alpha \in \mathcal{F}_{D}$ so $a$-sum in the right hand consists of $\ll_{\delta} \log(Tt)$ pieces. Since $T\ll_{D} t$, we therefore have the estimate
\begin{multline*}
(Tt)^{-\frac{1}{2}} \sum_{\psi\in \widehat{H^{+}(K)}}\sum_{n\in \mathbb{Z}} W_{k}(n) U(n)L(1/2,f\times\theta_{\chi_{\psi,n}}) \ll_{D,\varepsilon,\delta} \\ t^{\varepsilon} \sup_{0\leq 2^{a} \ll_{D} (Tt)^{1+\delta}} \Big| \sum_{m=0}^{(Tt)^{1+\delta}}\frac{\chi_{D}(m)}{m}  \sum_{\alpha \in \mathcal{F}_{D}}  \frac{\lambda_{f}(\alpha)}{\sqrt{|N(\alpha)|}}\Pi_{\alpha}(m) W_{a}^{'} (\alpha)U_{T}(\alpha) \Big|.
\end{multline*}
for any $\varepsilon>0$, $\delta>0$ and $1\leq T \ll_{D} t$.

\subsection{Poisson summation}
We now fix $0\leq a\leq  a_{T} $ and write $A\coloneqq 2^{a}$. We can define the smooth compactly supported function
\[
H_{A}(\alpha,m,x) \coloneqq \frac{(Tt)^{-\frac{1}{2}}}{\sqrt{|N(\alpha)|}} W_{a}^{'}(\alpha)U_{T}(\alpha)W_{k}(x)U(x) V_{\frac{1}{2}}\Big(\frac{m^{2}|N(\alpha)|}{D},x\Big) .
\]
Notice that 
$
\Big|\frac{\alpha}{\alpha^{*}}\Big|^{\frac{\pi i x}{\log(\varepsilon_{D})}} = e\Big(\frac{x \log(\frac{\alpha}{\alpha^{*}})}{2\log(\varepsilon_{D})}\Big)
$,
and we can therefore apply the Poisson summation formula to the $n$-sum.  Our problem then reduces to estimating sums of the form
\begin{equation}
\label{eq: poisson1}
      \sum_{m= 1}^{(Tt)^{1+\delta}} \frac{\chi_{D}(m)}{m}   \sum_{\alpha \in \mathcal{F}_{D} }  \lambda_{f}(\alpha) \sum_{\xi \in \mathbb{Z}} \widehat{H_{A}}\Big(\alpha,m,\xi-\frac{ \log(\frac{\alpha}{\alpha^{*}})}{2\log(\varepsilon_{D})}\Big)
\end{equation}
where the Fourier coefficients of $H_{A}(\alpha,m,x)$ are taken with respect to the $x$ variable. In particular, we need to estimate the Fourier coefficients $\widehat{H_{A}}(\alpha,m,\xi)$. For $\xi\neq 0$, a standard argument using integration by parts gives us the inequality
\begin{equation}
\label{eq: Fourier by parts}
\widehat{H_{A}}(\alpha,m,\xi) \leq \frac{1}{|\xi|^{j}} \Big| \widehat{H^{(j)}_{A}}(\alpha,m,\xi) \Big|
\end{equation}
where $H^{(j)}_{A}(\alpha,m,x )$ is the $j$-th partial derivative of $H_{A}(\alpha,m,x)$ with respect to $x$. We estimate the Fourier coefficients of $H^{(j)}_{A}(\alpha,m,x )$ using the following proposition:
\begin{proposition}
\label{derivatives H}
Let $j\geq 0$. Then, 
\[
H^{(j)}_{A}(\alpha,m,x) = 0 \text{ unless }  x \in \Big[c_{D}t-2T, c_{D}t- \frac{T}{2}] \text{ and } |N(\alpha)| \in \Big[\frac{A}{2}, 2A \Big].
\]
Moreover, for any $\varepsilon>0$ we have
\[ H^{(j)}_{A}(\alpha,m,x) \ll_{\varepsilon,j,D} t^{\varepsilon}T^{-j} (tTA)^{-\frac{1}{2}}.
\]
\end{proposition}
\begin{proof}
We know that $W_{a}^{'}(\alpha)$ is zero whenever $ |N(\alpha)| \notin \Big[ \frac{A}{2}, 2A \Big]$ and $W_{k}(x)$ is zero whenever $x \notin \Big[c_{D}t-2T, c_{D}t- \frac{T}{2}\Big]$. As a consequence, the first statement must hold.
\newline
\newline
We now prove the upper bound on the partial derivatives with respect to $x$. Of course, we may assume that $|N(\alpha)| \in \Big[ \frac{A}{2}, 2A \Big]$. Using the product rule for derivatives, we see that
\begin{align*}
H_{A}^{(j)}(\alpha,m,x) 
&\ll_{j} (TtA)^{-\frac{1}{2}} \sum_{i=0}^{j}  \Big(W_{k}(x)U(x)\Big)^{(j-i)} V^{(i)}_{\frac{1}{2}}\Big( \frac{m^{2}|N(\alpha)|}{D},x\Big)
\end{align*}
where again all partial derivatives are taken with respect to $x$.
Recall that $W_{k}(x) = W\Big(\frac{c_{D}t- x}{T}\Big)$ for some fixed smooth compactly supported function $W(x)$. Therefore, by repeatedly applying the chain rule we easily see that
$
W^{(j-i)}_{k}(x) \ll_{j,i} T^{-j+i}
$. Similarly, since $U(x) = \tilde{U}\Big( \frac{x}{c_{D}t}\Big)$ with $\tilde{U}(x)$ a fixed bump function independent of $t$ we have
$
U^{(j-i)}(x)  \ll_{D,i,j} T^{-j+i}.
$
Finally, by Proposition \ref{V derivative with respect to n}, we know that  
\begin{equation}
    V_{\frac{1}{2}}^{(i)} \Big( \frac{m^{2}|N(\alpha)|}{D}, x \Big) \ll_{\varepsilon,i} D^{\frac{\varepsilon}{4}} t^{\varepsilon}T^{-i}.
\end{equation}
for any $i\geq 0$. Combining everything, we get the desired estimate.
\end{proof}

\noindent As an immediate consequence of Proposition \ref{derivatives H} and (\ref{eq: Fourier by parts}) we get
\begin{corollary}
\label{fourier bounds}
For any $j\geq 0$ and $\xi \neq 0$, we have
\begin{equation*}
    \widehat{H_{A}}(\alpha,m,\xi) \ll_{\varepsilon,j,D} \frac{Tt^{\varepsilon}}{|T\xi|^{j}} (tTA)^{-\frac{1}{2}}.
\end{equation*}
Moreover for $|N(\alpha)| \notin  \Big[ \frac{A}{2}, 2A \Big]$, we have $\widehat{H_{A}}(\alpha,m,\xi) = 0$.
\end{corollary}
\noindent Using this estimate, we show that the sums of the form (\ref{eq: poisson1}) can essentially be restricted to finitely many $\xi \in \mathbb{Z}$. Moreover, for each such $\xi$, the only generators $\alpha \in \mathcal{F}_{D}$ which survive lie in a small cone. More formally, we have 
\begin{proposition}
\label{possion summation result}
For any $m\geq 1$ and $0< \delta < 1$, we have
\begin{multline}
\label{eq: poisson summation result}
\sum_{\alpha \in \mathcal{F}_{D} }  \lambda_{f}(\alpha) \sum_{\xi \in \mathbb{Z}} \widehat{H_{A}}\Big(\alpha,m,\xi-\frac{ \log(\frac{\alpha}{\alpha^{*}})}{2\log(\varepsilon_{D})}\Big) \\ \ll_{\varepsilon,\delta,D} t^{\varepsilon}T(TtA)^{-\frac{1}{2}}\sum_{\xi =-1}^{2} \sum_{\alpha \in S_{2\xi,2T^{-1+\delta}}} |\lambda_{f}(\alpha)| + t^{\varepsilon}T^{-\frac{3}{2}}(tA)^{-\frac{1}{2}}\sum_{\substack{\alpha \in \mathcal{F}_{D} \\ \frac{A}{2}\leq |N(\alpha)| \leq 2A}} |\lambda_{f}(\alpha)|
\end{multline}
where we define
\[
S_{\xi,R} \coloneqq \Big\{  (\alpha,\alpha^{*}) \in \mathcal{F}_{D} \mid \varepsilon_{D}^{\xi-R} \alpha^{*} \leq \alpha \leq \varepsilon_{D}^{\xi+R} \alpha^{*} \text{ , } \frac{A}{2} \leq |N(\alpha)| \leq 2A \Big\}.
\]
for any $R>0$, $\xi \in \mathbb{Z}$. By abuse of notation, we write $\alpha \in S_{\xi,R}$ instead of $(\alpha,\alpha^{*}) \in S_{\xi,R}$.
\end{proposition}
\begin{proof}
By definition, for any $\alpha \in \mathcal{F}_{D}$ we have
$
0 \leq \frac{\log(\frac{\alpha}{\alpha^{*}})}{\log(\varepsilon_{D})} < 2 
$.
Therefore, by Corollary \ref{fourier bounds} we have
\begin{align*}
   \sum_{\alpha \in \mathcal{F}_{D}} \lambda_{f}(\alpha) \sum_{\xi \geq 3} \widehat{H_{A}}\Big(\alpha,m,\xi-\frac{ \log(\frac{\alpha}{\alpha^{*}})}{2\log(\varepsilon_{D})}\Big)  &\ll_{\varepsilon,D} t^{\varepsilon}T^{-1}(tTA)^{-\frac{1}{2}} \sum_{\substack{\alpha \in \mathcal{F}_{D} \\ \frac{A}{2}\leq |N(\alpha)| \leq 2A}} |\lambda_{f}(\alpha)|\sum_{\xi \geq 3} \frac{1}{\Big|\xi-\frac{ \log(\frac{\alpha}{\alpha^{*}})}{2\log(\varepsilon_{D})}\Big|^{2}} \\
   &\ll_{\varepsilon,D} t^{\varepsilon}T^{-1}(tTA)^{-\frac{1}{2}} \sum_{\substack{\alpha \in \mathcal{F}_{D} \\ \frac{A}{2}\leq |N(\alpha)| \leq 2A}} |\lambda_{f}(\alpha)|.
\end{align*}
Similarly, 
\begin{align*}
    \sum_{\alpha \in \mathcal{F}_{D}} \lambda_{f}(\alpha) \sum_{\xi \leq -2 } \widehat{H_{A}}\Big(\alpha,m,\xi-\frac{ \log(\frac{\alpha}{\alpha^{*}})}{2\log(\varepsilon_{D})}\Big)  &\ll t^{\varepsilon}T^{-1}(tTA)^{-\frac{1}{2}}  \sum_{\substack{\alpha \in \mathcal{F}_{D} \\ \frac{A}{2}\leq |N(\alpha)| \leq 2A}} |\lambda_{f}(\alpha)| .
\end{align*}
Let now $\xi \in \{-1,0,1,2\}$ and $\alpha \in \mathcal{F}_{D}$.  One easily sees that $|\xi - \frac{ \log(\frac{\alpha}{\alpha^{*}})}{2\log(\varepsilon_{D})}|\leq T^{-1+\delta}$ if and only if
$
\alpha \in S_{2\xi, 2T^{-1+\delta}}
$
in which case we use the estimate
\[
  \widehat{H_{A}}\Big(\alpha,m,\xi-\frac{ \log(\frac{\alpha}{\alpha^{*}})}{2\log(\varepsilon_{D})}\Big) \ll_{\varepsilon,D} t^{\varepsilon}T(tTA)^{-\frac{1}{2}}
\]
from Corollary \ref{fourier bounds}.
On the other hand, if $\alpha \in \mathcal{F}_{D}$ satisfies $|\xi - \frac{ \log(\frac{\alpha}{\alpha^{*}})}{2\log(\varepsilon_{D})}|> T^{-1+\delta}$, then by Corollary \ref{fourier bounds} we have 
\begin{align*}
    \widehat{H_{A}}\Big(\alpha,m,\xi-\frac{ \log(\frac{\alpha}{\alpha^{*}})}{2\log(\varepsilon_{D})}\Big) \ll_{\varepsilon,\delta,D} \frac{t^{\varepsilon}T^{1-\delta j}}{\Big|T^{1-\delta}\Big(\xi - \frac{ \log(\frac{\alpha}{\alpha^{*}})}{2\log(\varepsilon_{D})}\Big)\Big|^{j}} (tTA)^{-\frac{1}{2}} \ll_{\varepsilon,\delta,D} t^{\varepsilon}T^{-1} (tTA)^{-\frac{1}{2}}
\end{align*}
for $j= \lceil \frac{2}{\delta} \rceil$.
Combining everything and rearranging the order of summation, we get the expected estimate.
\end{proof}

\noindent Notice that the right hand side of (\ref{eq: poisson summation result}) does not depend on $m$, so we can estimate the sum over $m$ separately using the well known-bounds for the harmonic sum.
As a result, if $\lambda_{f}(n) \ll_{\varepsilon} n^{\theta+\varepsilon}$ is any (uniform) bound towards the Ramanujan conjecture, we have
\begin{equation}
\label{eq: after poisson equation}
T(TtA)^{-\frac{1}{2}}   \sum_{\xi =-1}^{2} \sum_{\alpha \in S_{2\xi,2T^{-1+\delta}}} |\lambda_{f}(\alpha)|\sum_{m= 1}^{(Tt)^{1+\delta}} \frac{1}{m} \ll_{\varepsilon,\delta, D} t^{\varepsilon} A^{\theta +\varepsilon}T(TtA)^{-\frac{1}{2}} \Big| S_{2\xi,2T^{-1+\delta}} \Big|
\end{equation}
and
\begin{equation}
\label{eq: after poisson equation 2}
T^{-1}(TtA)^{-\frac{1}{2}}  \sum_{\substack{\alpha \in \mathcal{F}_{D} \\ \frac{A}{2}\leq |N(\alpha)| \leq 2A}} |\lambda_{f}(\alpha)|\sum_{m= 1}^{(Tt)^{1+\delta}} \frac{1}{m} \ll_{\varepsilon,\delta, D} t^{\varepsilon} A^{\theta +\varepsilon}T^{-1}(TtA)^{-\frac{1}{2}} \sum_{\substack{\alpha \in \mathcal{F}_{D} \\ \frac{A}{2}\leq |N(\alpha)| \leq 2A}} 1.
\end{equation}
But for any $n\geq 1$, one can easily show that $\sum_{\mathbb{N}(\mathfrak{a})=n } 1 \ll_{\varepsilon} n^{\varepsilon}$ where the sum is taken over all ideals $\mathfrak{a}$ of $\mathcal{O}_{K}$ with norm equal to $n$. For this reason, (\ref{eq: after poisson equation 2}) leads to
\[
T^{-1}(TtA)^{-\frac{1}{2}}  \sum_{\substack{\alpha \in \mathcal{F}_{D} \\ \frac{A}{2}\leq |N(\alpha)| \leq 2A}} |\lambda_{f}(\alpha)|\sum_{m= 1}^{(Tt)^{1+\delta}} \frac{1}{m} \ll_{\varepsilon,\delta, D}  t^{2(\theta + \varepsilon)}T^{-1}.
\]
\noindent We conclude the proof by trivially estimating the cardinality of the sets $S_{\xi,R}$ for any $\xi \in \{-1,0,1,2\}$ and $0< R \leq 2$.

\subsection{Trivial estimates}

\noindent To estimate $|S_{\xi,R}|$, we will first reduce our problem to counting lattice points in a parallelogram then apply the Lipschitz principle. 
We start by introducing the parallelogram
\[
\mathcal{P}_{\xi,R} \coloneqq \Big\{ a(\varepsilon_{D} \sqrt{2A}, \varepsilon_{D}^{1-\xi-R}\sqrt{2A}) + b(0,2\varepsilon_{D}^{5}R\sqrt{2A}) \mid a,b \in [0,1] \Big \}.
\]
To simplify the notation, for $\xi$ and $R$ fixed we define
\[
v_{1} = (\varepsilon_{D} \sqrt{2A}, \varepsilon_{D}^{1-\xi-R}\sqrt{2A}) \text{ and } v_{2} = (0,2\varepsilon_{D}^{5}R\sqrt{2A}).
\]
\begin{lemma}
\label{parallelogram lattice points}
For $\xi \geq -1$ and $0<R\leq 2$ we have 
\[
S_{\xi,R} \subseteq \mathcal{P}_{\xi,R} \cap L_{D}
\]
where $L_{D}$ is the lattice defined in Proposition \ref{fundamental domain}.
\end{lemma}
\begin{proof}
Since $S_{\xi,R}$ is a subset of the fundamental domain $\mathcal{F}_{D}$ which itself is a subset of $L_{D}$, it is enough to prove that 
$S_{\xi,R} \subseteq \mathcal{P}_{\xi,R}.$
\newline 
\newline 
Given any vector $(x,y) \in \mathbb{R}^{2}$, we can write
\[
(x,y) = \frac{x}{\varepsilon_{D}\sqrt{2A}}v_{1} + \frac{y-\varepsilon_{D}^{-\xi-R}x}{2\varepsilon_{D}^{5}R\sqrt{2A}} v_{2}
\]
so that
\begin{equation}
\label{eq: parallelogram conditions}
    (x,y) \in \mathcal{P}_{\xi,R} \iff \begin{cases}
    0\leq x \leq \varepsilon_{D}\sqrt{2A} \\ \varepsilon_{D}^{-\xi-R}x \leq y \leq \varepsilon_{D}^{-\xi-R}x  + 2\varepsilon_{D}^{5}R\sqrt{2A}.
    \end{cases}
\end{equation}
Hence, for any $ (x,y) \in S_{\xi,R}$, it is enough to prove that $x$ and $y$ satisfy the inequalities above. By definition of $\mathcal{F}_{D}$, we know that $\varepsilon_{D}^{-2} x \leq y \leq x$. The inequality 
$
0  \leq x \leq \varepsilon_{D} \sqrt{2A} 
$ now follows easily from $N(x) = xy$.
\newline
\newline 
On the other hand, the condition $\varepsilon_{D}^{\xi- R} y \leq x \leq \varepsilon_{D}^{\xi +R} y$ from the definition of $S_{\xi,R}$ is equivalent to
$\varepsilon_{D}^{-\xi-R} x \leq y \leq \varepsilon_{D}^{-\xi +R}x.$
By the mean value theorem applied to the function $a \mapsto \varepsilon_{D}^{a}$, we have
\begin{align*}
    \varepsilon_{D}^{-\xi+R}x &= \varepsilon_{D}^{-\xi-R}x + (\varepsilon_{D}^{-\xi +R} - \varepsilon_{D}^{-\xi -R})x \leq \varepsilon_{D}^{-\xi-R}x + 2\varepsilon_{D}^{2-\xi+R}R \sqrt{2A}.
\end{align*}
Since we assumed $-1\leq \xi$ and $R\leq 2$, we must have $\varepsilon_{D}^{-\xi+R} \leq \varepsilon_{D}^{3}$, which concludes the proof.
\end{proof}

\noindent As a result, we can apply Corollary \ref{lattice points parallelogram bounds} using the basis $\mathcal{B} = \{(1,1), (\beta_{D},\beta_{D}^{*})\}$ of $L_{D}$ to show that for any $\xi \geq -1$ and $0< R\leq 2$ we have
\[
|S_{\xi,R}| \ll_{D} RA+ \sqrt{A}.
\]

\noindent Hence, for any $\varepsilon>0$, $T\ll_{D} t$ and $A \ll_{\varepsilon} (Tt)^{1+\varepsilon}$, (\ref{eq: after poisson equation}) becomes 
\begin{align*}
    T(TtA)^{-\frac{1}{2}}   \sum_{\xi =-1}^{2} \sum_{\alpha \in S_{2\xi,2T^{-1+\varepsilon}}} \lambda_{f}(\alpha)\sum_{m= 1}^{(Tt)^{1+\varepsilon}} \frac{1}{m} &\ll_{D,\varepsilon} t^{3\varepsilon} A^{\theta} T(TtA)^{-\frac{1}{2}} (T^{-1+\varepsilon}A + A^{\frac{1}{2}}) \ll_{D,\varepsilon} t^{5\varepsilon} t^{2\theta}.
\end{align*}
This shows that, as expected, for any $\varepsilon>0$, we have
\[
||f|_{\mathcal{C}_{D}}||_{2}^{2} \ll_{D,\varepsilon} t^{2(\theta+\varepsilon)}
\]
where $\lambda_{f}(n) \ll_{\varepsilon} n^{\theta+\varepsilon}$ is any bound towards the Ramanujan conjecture.

%
%
%


\bibliographystyle{plainnat}
\bibliography{bibliography}

\end{document}